\documentclass{amsart} 

\usepackage{amsmath,amssymb,amsthm} 
\usepackage{graphicx} 
\usepackage{subfig}
\usepackage[arrow, matrix, curve]{xy} 
\usepackage{color} 
\usepackage{float} 
\usepackage{makecell} 
\usepackage{multirow} 
\usepackage{changepage}
\usepackage{hyperref} 
\usepackage{enumerate}
\usepackage{booktabs} 
\usepackage{tabularx}
\usepackage{ltablex}
\usepackage{array}
\usepackage{siunitx} 
\usepackage{fullpage}

\newcommand{\Z}{\mathbb{Z}}

\newcommand{\N}{\mathbb{N}}
\newcommand{\C}{\mathbb{C}}

\newcommand{\spin}{\ifmmode{\rm Spin}\else{${\rm spin}$\ }\fi}
\newcommand{\spinc}{\ifmmode{{\rm Spin}^c}\else{${\rm spin}^c$}\fi}

\DeclareMathOperator{\lk}{lk}

\makeatletter
\newcommand{\Vast}{\bBigg@{2.5}} 
\makeatother

\newtheoremstyle{thm}{}{}{\itshape}{}{\bfseries}{}{ }{} 
\newtheoremstyle{definition}{}{}{}{}{\bfseries}{}{ }{} 

\theoremstyle{thm}
\newtheorem{Theorem}{Theorem}[section]
\newtheorem{theorem}[Theorem]{Theorem}
\newtheorem{lemma}[Theorem]{Lemma}
\newtheorem{proposition}[Theorem]{Proposition}

\newtheorem*{Theorem-ohne}{Theorem}

\newtheorem{question}[Theorem]{Question}

\theoremstyle{definition}

\newtheorem{rem}[Theorem]{Remark}

\begingroup\expandafter\expandafter\expandafter\endgroup
\expandafter\ifx\csname pdfsuppresswarningpagegroup\endcsname\relax
\else
\fi

\definecolor{amaranth}{rgb}{0.9, 0.17, 0.31} 
\definecolor{carrotorange}{rgb}{0.93, 0.57, 0.13} 
\definecolor{citrine}{rgb}{0.89, 0.82, 0.04} 
\definecolor{dartmouthgreen}{rgb}{0.05, 0.5, 0.06} 
\definecolor{ballblue}{rgb}{0.13, 0.67, 0.8} 
\definecolor{ceruleanblue}{rgb}{0.16, 0.32, 0.75} 
\definecolor{amethyst}{rgb}{0.6, 0.4, 0.8} 
\definecolor{amber}{rgb}{1.0, 0.75, 0.0} 
\definecolor{burlywood}{rgb}{0.87, 0.72, 0.53} 

\numberwithin{equation}{section}

\begin{document}

\title{Census L-space knots are braid positive, except for one that is not.}

\author{Kenneth L. Baker}
\address{Department of Mathematics, University of Miami, Coral Gables, FL 33146, USA}
\email{k.baker@math.miami.edu}

\author{Marc Kegel}
\address{Humboldt-Universit\"at zu Berlin, Rudower Chaussee 25, 12489 Berlin, Germany.}
\email{kegemarc@math.hu-berlin.de, kegelmarc87@gmail.com}


\date{\today} 

\begin{abstract}
We exhibit braid positive presentations for all L-space knots in the SnapPy census except one, which is not braid positive.   The normalized HOMFLY polynomial of $o9\_30634$, when suitably normalized is not positive, failing a condition of Ito for braid positive knots.

We generalize this knot to a 1-parameter family of hyperbolic L-space knots that might not be braid positive.  Nevertheless, as pointed out by Teragaito, this family yields the first examples of hyperbolic L-space knots whose formal semigroups are actual semigroups, answering a question of Wang. 
Furthermore, the roots of the Alexander polynomials of these knots are all roots of unity, disproving a conjecture of Li-Ni.
\end{abstract}


\makeatletter
\@namedef{subjclassname@2020}{%
  \textup{2020} Mathematics Subject Classification}
\makeatother


\maketitle


\section{Introduction}
Based on observation, most L-space knots are braid positive.  Here, {\em L-space knots} are knots in $S^3$ with a {\em positive} Dehn surgery to an L-space~\cite{Ozsvath2005}, and a knot that is the closure of a positive braid is  {\em braid positive}. The L-space torus knots are the positive torus knots, and hence they are braid positive.  Notably however, the $(2,3)$--cable of the $(2,3)$--torus knot is an L-space knot~\cite{Hedden2005} that is not braid positive, see e.g.~\cite[Table 8]{Dunfield2019} and~\cite[Example 1]{anderson2021lspace}.  It stands to reason that there probably are other cable L-space knots which are not braid positive.  Nevertheless, it was questioned if every hyperbolic L-space knot is braid positive, e.g.~\cite[Problem 31.2]{Hom2017}.

Dunfield shows that there are exactly 1,267 complements of knots in $S^3$ in the SnapPy census of $1$--cusped hyperbolic manifolds that can be triangulated with at most $9$ ideal tetrahedra~\cite{Dunfield2018}.   He further determines that (up to mirroring) 635 are not L-space knots, 630 are L-space knots, and left 2 as undetermined~\cite{Dunfield2019}.  These last 2 have been shown to have quasi-alternating surgeries~\cite{BKM} and hence they are L-space knots as well.  Thus there are exactly 632 L-space knots in the SnapPy census. 

\begin{theorem}\label{thm:braidpositivecensus}
Every L-space knot in the SnapPy census of up to $9$ tetrahedra is braid positive except for $o9\_30634$ which is not.   
\end{theorem}

\begin{figure}
    \centering
    \includegraphics[width=0.85\textwidth]{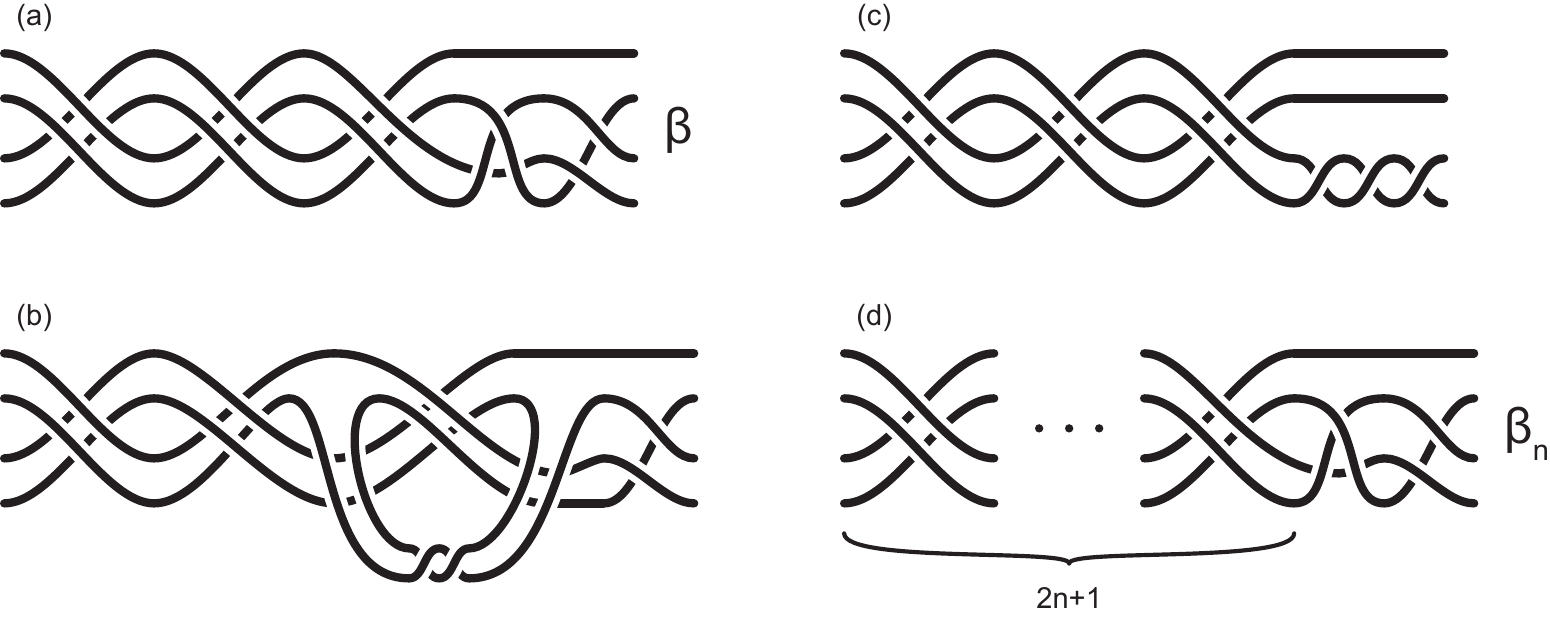}
    \caption{(a) The braid $\beta$ is positive except for one strongly quasipositive crossing.  Its closure $\widehat{\beta}$ is the hyperbolic L-space knot $o9\_30634$ which we show is not braid positive.  (b) Dragging the base of the strongly quasipositive band of $\beta$ into the position shown exhibits  $\widehat{\beta}$ as a positive Hopf basket.  (c) This braid has the $(2,3)$--cable of the $(2,3)$--torus knot as its closure.  (d) The closures of the braids $\beta_n$ are L-space knots that may also fail to be braid positive.}    \label{fig:braids}
\end{figure}
The knot $o9\_30634$ is {\em nearly braid positive} in the sense that  it has a braid presentation that is braid positive except for one strongly quasipositive crossing that jumps over only one strand. We do not know if $o9\_30634$ admits a positive diagram.

\begin{question}
Is every hyperbolic L-space knot nearly braid positive?
\end{question}

\begin{proof} [Proof of Theorem~\ref{thm:braidpositivecensus}]
In the work~\cite{BKM} we obtained braid words for every census L-space knot by automating the process from~\cite{anderson2021lspace}. (An alternative approach is taken in~\cite{RDO21}.)  Here, utilizing the braid and simplification methods in SnapPy \cite{snappy} / Sage \cite{sagemath}, we managed to cajole braid positive presentations for all of the knots except for one, $o9\_30634$. The L-space census knots and positive braids with them as closures are detailed in the Appendix~\ref{appendix} and verified in~\cite{BK}.

As one may check, the knot $K=o9\_30634$ is the closure of the $4$--braid
\[\beta = [2,1,3,2,2,1,3,2,2,1,3,2,-1,2,1,1,2]. \]
Here the list of non-zero integers represents a braid word by letting the integer $k$ stand for the standard generator $\sigma_k$ or its inverse $\sigma_{k}^{-1}$ depending on whether $k$ is positive or negative.

Ito gives new constraints on a suitably normalized version of the HOMFLY polynomial for positive braids~\cite{ito2020note}.
The Ito-normalized HOMFLY polynomial  $\widetilde{P_K}(\alpha,z) = \sum h_{i\, j} \alpha^i z^{2j}$ of $K = \widehat{\beta}$ is represented by the matrix $H= (h_{i\, j})$ of coefficients
\[ H= 
\left(\begin{array}{rrrrrrr}
13 & 69 & 133 & 121 & 55 & 12 & 1 \\
17 & 66 & 83 & 45 & 11 & 1 & 0 \\
4 & 10 & 6 & 1 & 0 & 0 & 0 \\
-1 & -1 & 0 & 0 & 0 & 0 & 0
\end{array}\right)
\]
where the indexing starts at $0\, 0$ so that $h_{0\, 0} = 13$.  One may calculate this with Sage (or the knot theory package \cite{knotatlas} for Mathematica \cite{Mathematica}) from the braid word using the built-in HOMFLY polynomial and adjusting it to achieve Ito's normalization. The computations can be found at~\cite{BK}.

According to~\cite[Theorem 2]{ito2020note}, if a link $K$ is braid positive then the Ito-normalized HOMFLY polynomial should only have non-negative coefficients.  As one observes, the coefficients $h_{3\, 0}$ and $h_{3\, 1}$ are negative.  Hence $o9\_30634$ is not braid positive.
\end{proof}

In Section~\ref{sec:potential_counterexamples}, we generalize the knot $o9\_30634$ to an infinite family of hyperbolic L-space knots that are nearly braid positive but for which Ito's constraints fail to obstruct braid positivity, at least for the examples we managed to calculate.   In Section~\ref{sec:generalized}, we further extend this family to a doubly infinite family of knots $K_{n,m}$ in hopes of providing more potential examples.  While that doesn't quite work out, we highlight several properties of these knots in Proposition~\ref{prop:Knm}.  Notably, we show that
\begin{itemize}
    \item all but $K_{-1,m}$ and further six exceptional cases of these knots are hyperbolic, 
    \item identify a small Seifert fibered space surgery for each, 
    \item determine that when $n\geq 0$ they are L-space knots if and only if $m \leq 0$,  
    \item compute their Alexander polynomials,  and 
    \item examine their structures as positive braids and strongly quasipositive braids.
\end{itemize}
Lastly, in Section~\ref{sec:curiosities} we observe that our infinite family of hyperbolic L-space knots of Section~\ref{sec:potential_counterexamples} have Alexander polynomials that 
\begin{itemize}
    \item induce formal semigroups that are actually semigroups (which Teragaito pointed out to us) and
    \item have all their roots on the unit circle, disproving \cite[Conjecture 1.3]{LiNi}.
\end{itemize}

After the publication of this paper, this conjecture appeared as Problem 1.21(c)(iii) in the K3 problem list~\cite{KirbyK3}. Thus, the examples constructed here give a negative solution to that problem.

\section{A family of hyperbolic L-space knots that might not be braid positive}\label{sec:potential_counterexamples}

Let $\{K_n\}$ be the family of knots that are the closures of the braids 
\[\beta_n = [(2,1,3,2)^{2n+1}, -1,2,1,1,2]. \]
and includes our knot $o9\_30634$ as $K_1$. See Figure~\ref{fig:braids}(d). Observe that $\beta_n$ gives a strongly quasi-positive braid presentation for these knots that is {\em almost} braid positive, it is braid positive except for one negative crossing. 

\begin{proposition}
For $n\geq1$, the knots $K_n$ are hyperbolic L-space knots.
\end{proposition}


\begin{proof}
This follows from Lemmas~\ref{lem:Lspaceknots} and \ref{lem:hyperbolicknots} below.
\end{proof}

\begin{lemma}\label{lem:Lspaceknots}
For $n\geq1$, the knots $K_n$ are L-space knots.  In particular, the  $8n+6$ surgery on $K_n$ gives the Seifert fibered L-space
$M(-1;  \frac{1}{2}, \frac{2n+1}{4n+4}, \frac{2}{4n+5})$.
\end{lemma}

\begin{proof}
Figure~\ref{fig:surgdesc} shows how a strongly invertible surgery description of the knot $K_n$ along with its $8n+6$ surgery may be obtained.  Figure~\ref{fig:monty} demonstrates how one may take the quotient and perform rational tangle replacements associated to the surgeries to produce a link whose double branched cover is $8n+6$ surgery on $K_n$. We observe this link to be the Montesinos link $M(\frac{2}{4n+5}, \frac{1}{2}, -\frac{2n+3}{4n+4})$.   Hence its double branched cover is the Seifert fibered space $M_n(0; \frac{2}{4n+5}, \frac{1}{2}, -\frac{2n+3}{4n+4})$.  Here we use the notation of Lisca--Stipsicz~\cite{LS}
where the Seifert fibered space $M(e_0;r_1, r_2, \dots, r_k)$ is obtained by $e_0$--surgery on an unknot with $k$ meridians having $-1/r_i$ surgery on the $i$th one.

\begin{figure}
    \centering
    \includegraphics[width=\textwidth]{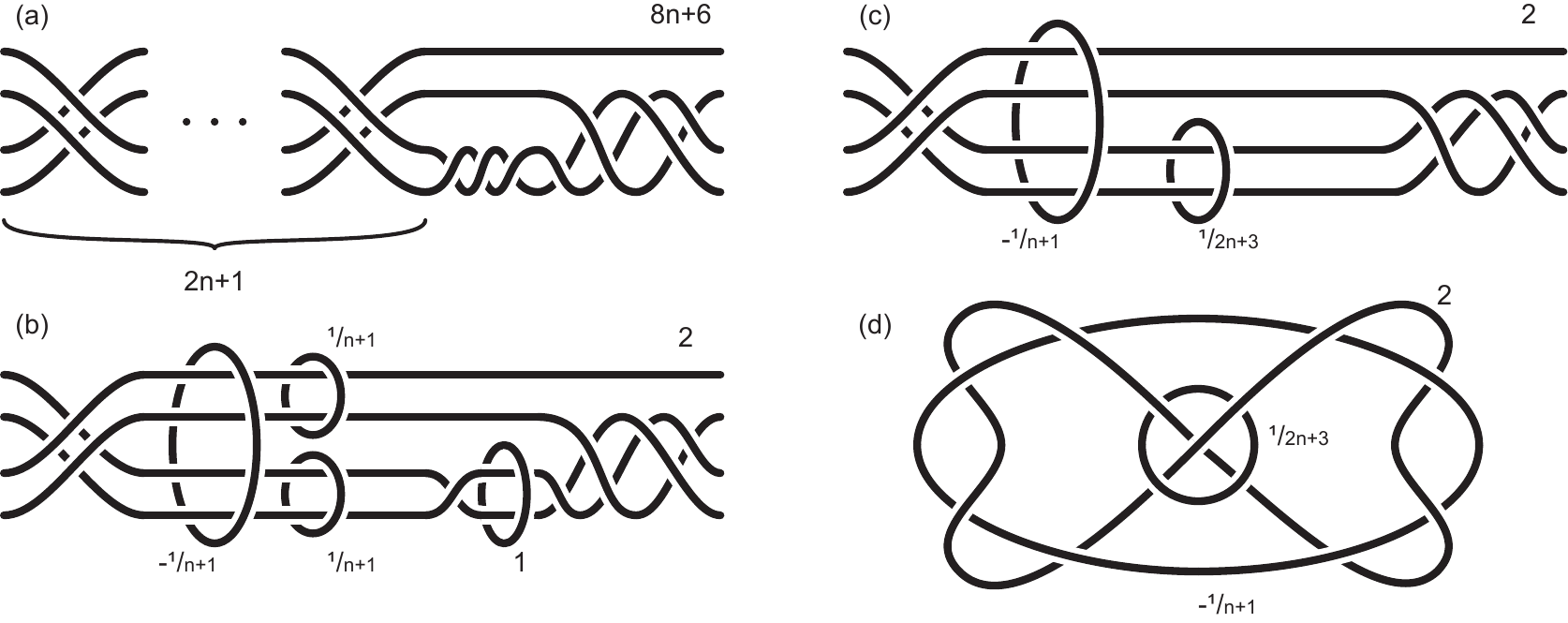}
    \caption{(a) The braid $\beta_n$ with a surgery coefficient of $8n+6$ for its closure knot $K_n$. (b) \& (c) Twists in the braid are expressed and collected into surgeries on unknots.  The surgery coefficient on the closure knot is adjusted accordingly. (d) After closure and isotopy, we obtain a surgery description for $8n+6$ surgery on $K_n$.}
    \label{fig:surgdesc}
\end{figure}

\begin{figure}
    \centering
    \includegraphics[width=\textwidth]{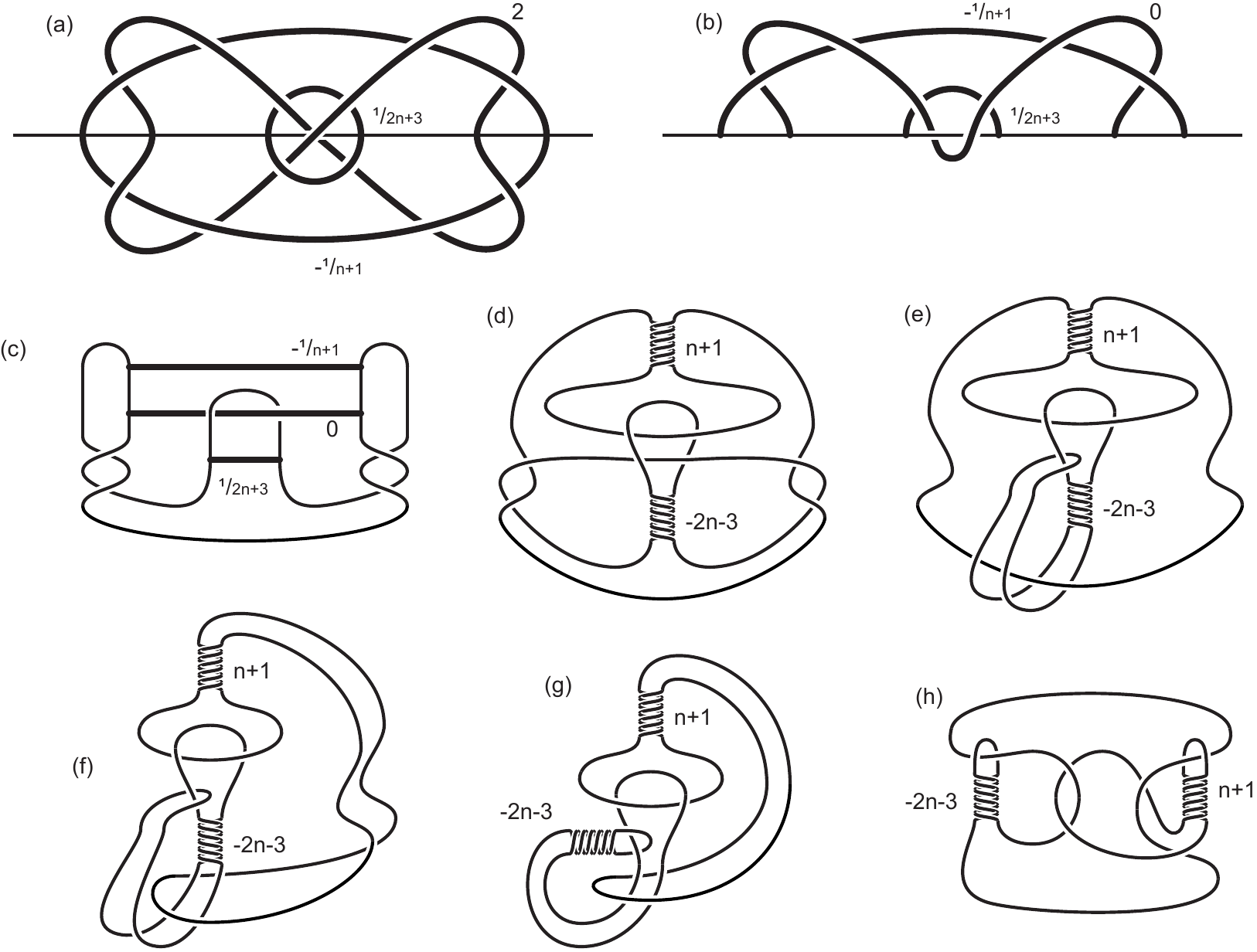}
    \caption{(a) The surgery description from Figure~\ref{fig:surgdesc}(d) is strongly invertible. (b)\&(c) The quotient of the surgery description followed by some isotopy to straighten the arcs. (d) Rational tangle replacements along the arcs produce a link whose double branched cover is $8n+4$ surgery on $K_n$. (e)--(h) A sequence of isotopies show this link is the Montesinos link $M([0,-2n-3,-2],[0,-2], [0,1,-1,n+1,2]) = M(\frac{2}{4n+5}, \frac{1}{2}, -\frac{2n+3}{4n+4})$.}
    \label{fig:monty}
\end{figure}

These Seifert fibered spaces are determined to be L-spaces via~\cite[Theorem~1]{LS}.  
More specifically, 
Lisca--Stipsicz~\cite[Theorem 1]{LS} shows the Seifert fibered space $M=M(e_0;r_1, r_2, r_3)$ (with $1 \geq r_1\geq r_2\geq r_3 \geq 0$) is an L-space if and only if either $M$ or $-M$ does not carry a positive transverse contact structure. Then by Lisca--Mati\'c~\cite{LM}, such a Seifert fibered space $M$ carries no positive transverse contact structure if and only if either $e_0\geq0$ or $e_0=-1$ and there exists no coprime integers $a$ and $m$ such that $mr_1<a<m(1-r_2)$ and $mr_3<1$.

Rewriting to apply~\cite[Theorem 1]{LS}, we obtain that $M_n = M(-1;  \frac{1}{2}, \frac{2n+1}{4n+4}, \frac{2}{4n+5})$.
Then, since $1-r_2 = \frac{2n+3}{4n+4}$, we assume for contradiction that there are coprime integers $a$ and $m$ such that $m\frac12<a<m\frac{2n+3}{4n+4}$ and $m \frac{2}{4n+5} <1$.
The first gives
\[      0 < 2a-m < \frac{m}{2n+2}.   \]
The second implies $m < 2n+2 +\frac12$ so that $m \leq 2n+2$ and 
\[\frac{m}{2n+2} \leq 1.  \]
Together, they yield $0 < 2a - m <1$.
However, since $2a-m$ is an integer, there are no pairs of integers $a,m$ that satisfy this equation.  This is a contradiction.

Therefore $M_n$ does not carry a positive transverse contact structure, and thus it is an L-space.  Hence $K_n$ is an L-space knot for each $n \geq 1$.


\end{proof}

\begin{lemma}\label{lem:hyperbolicknots}
For $n\geq1$ the knots $K_n$ are hyperbolic.
\end{lemma}

\begin{proof} 


We check that $L12n1739(1,2n+2)(0,0)(-1,n+1)$ has the same exterior as $K_n$. Via SnapPy we verify that $L12n1739$ is hyperbolic and compute its short slopes of length less than $2\pi$ as
\begin{align*}
    &[(1, 0), (-2, 1), (-1, 1), (0, 1), (1, 1), (-1, 2), (1, 2), (-1, 3)],\\
 &[(1, 0), (-5, 1), (-4, 1), (-3, 1), (-2, 1), (-1, 1), (0, 1), (1, 1),(-5, 2),
  (-3, 2)],\\
 &[(1, 0), (-2, 1), (-1, 1), (0, 1), (1, 1), (2, 1), (-1, 2), (1, 2)].
\end{align*}
Thus for $n>1$ we fill with slopes longer than $2\pi$ and therefore directly get hyperbolic manifolds by Gromov's and Thurston's $2\pi$-Theorem (see for example~\cite[Theorem~9]{BleilerHodgson96}).
\end{proof}

Teragaito suggested an alternative approach to this lemma that does not use SnapPy or any computer calculation~\cite{teragaito}.  The referee also proposed a similar approach.  Since it is more ``hands-on'', we include a proof along the lines of their suggestions here.

\begin{proof}[Another proof of Lemma~\ref{lem:hyperbolicknots}]
As knots in $S^3$ are either torus knots, satellite knots, or hyperbolic knots by \cite{Thurston}, we must show that $K_n$ is neither a torus knot or a satellite knot.

In the proof of Theorem~\ref{thm:rootsofunity}
the Alexander polynomial of $K_n = K_{n,0}$ is presented as
\[  \Delta_{K_{n,0}} \doteq \frac{(t^{4n+5}+1)(t^{4n+2}+1)}{(t+1)(t^2+1)}.\]
As this is not equivalent to the Alexander polynomial of a torus knot, $K_n$ cannot be a torus knot. (Also, the formal semigroup of $K_n$ has rank $3$ as noted in Remark~\ref{rem:semigroup}, whereas the formal semigroup of a torus knot has rank $2$.)

So now suppose $K_n$ is a satellite knot.
Observe that an unknotting tunnel put at the unique negative crossing for $K_n = \widehat{\beta_n}$ in Figure~\ref{fig:braids}(d) shows that $K_n$ has tunnel number one.
Since the bridge index of $K_n$ is at most $4$, 
then Morimoto-Sakuma’s classification of tunnel number one satellite knots~\cite{morimoto-sakuma}  tells us that $K_n$ has the $2$--bridge torus knot $T(2,q)$ as a companion knot for some odd $q$ and a pattern of wrapping number $2$.
As $K_n$ is an L-space knot by Lemma~\ref{lem:Lspaceknots}, this pattern must be a braided pattern by \cite[Lemma 1.17]{BakerMotegi2019}.  Hence the pattern must be a $2$--cable. Thus if $K_n$ is a satellite knot, then it is a $2$--cable knot of $T(2,q)$.  Indeed, the Alexander polynomial of $K_n$ shown above implies that $K_n$ must be the $(2,4n+5)$--cable of the $T(2,2n+1)$ torus knot.  However, the distance of the cabling slope $8n+10$ and the slope $8n+6$ of the Seifert fibered surgery on $K_n$ is $\Delta(8n+10,8n+6) = 4$ is greater than $1$.  Thus the cabling torus remains incompressible after surgery, e.g.\ \cite[Lemma 7.2]{Gordon}. This contradicts that $8n+6$ surgery on $K_n$ produces a small Seifert fibered space.  Thus $K_n$ cannot be a satellite knot.  
\end{proof}


However the constraints of Ito on Homfly polynomials appear to not obstruct $K_n$ from being braid positive when $n\geq 2$.  Using Sage for computations, we see that Ito's constraints on the HOMFLY polynomials of $K_n$ for $n =2, \dots, 10$ do not obstruct braid positivity for these knots.  
Furthermore, we have been unsuccessful in finding a braid positive presentation for these knots.

\begin{question}
Are the knots $K_n$ for $n\geq2$ braid positive?
\end{question}

\section{A doubly infinite family of knots}\label{sec:generalized}

\begin{figure}
    \centering
    \includegraphics[width=\textwidth]{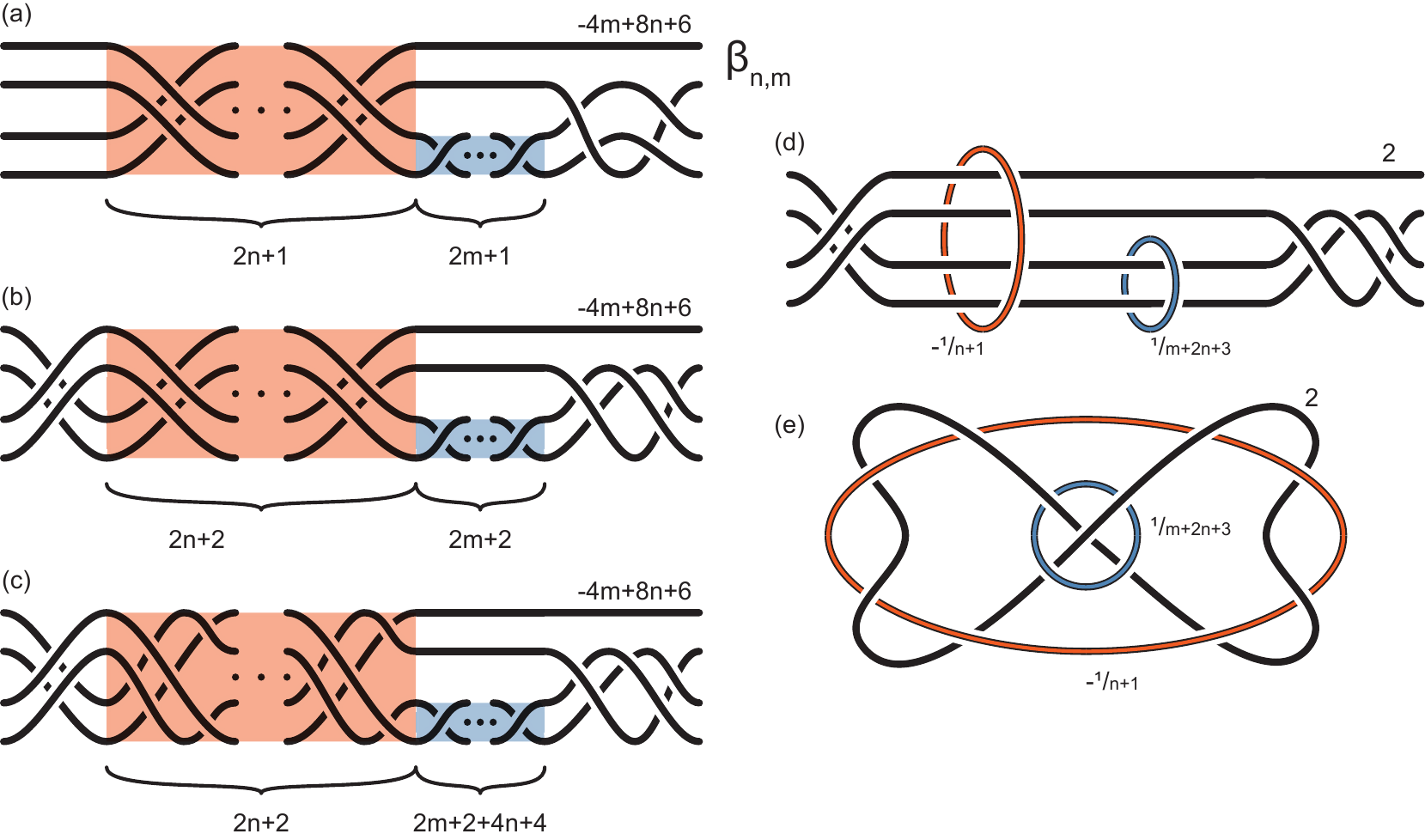}
    \caption{(a) The braid $\beta_{n,m}$ with a surgery coefficient of $-4m+8n+6$ for its closure knot $K_{n,m}$. (b) \& (c) \& (d) Twists in the braid are expressed and collected into surgeries on unknots.  The surgery coefficient on the closure knot is adjusted accordingly. (e) After closure and isotopy, we obtain a surgery description for $-4m+8n+6$ surgery on $K_{n,m}$.}
    \label{fig:gen-surgdesc}
\end{figure}

\begin{figure}
    \centering
    \includegraphics[width=\textwidth]{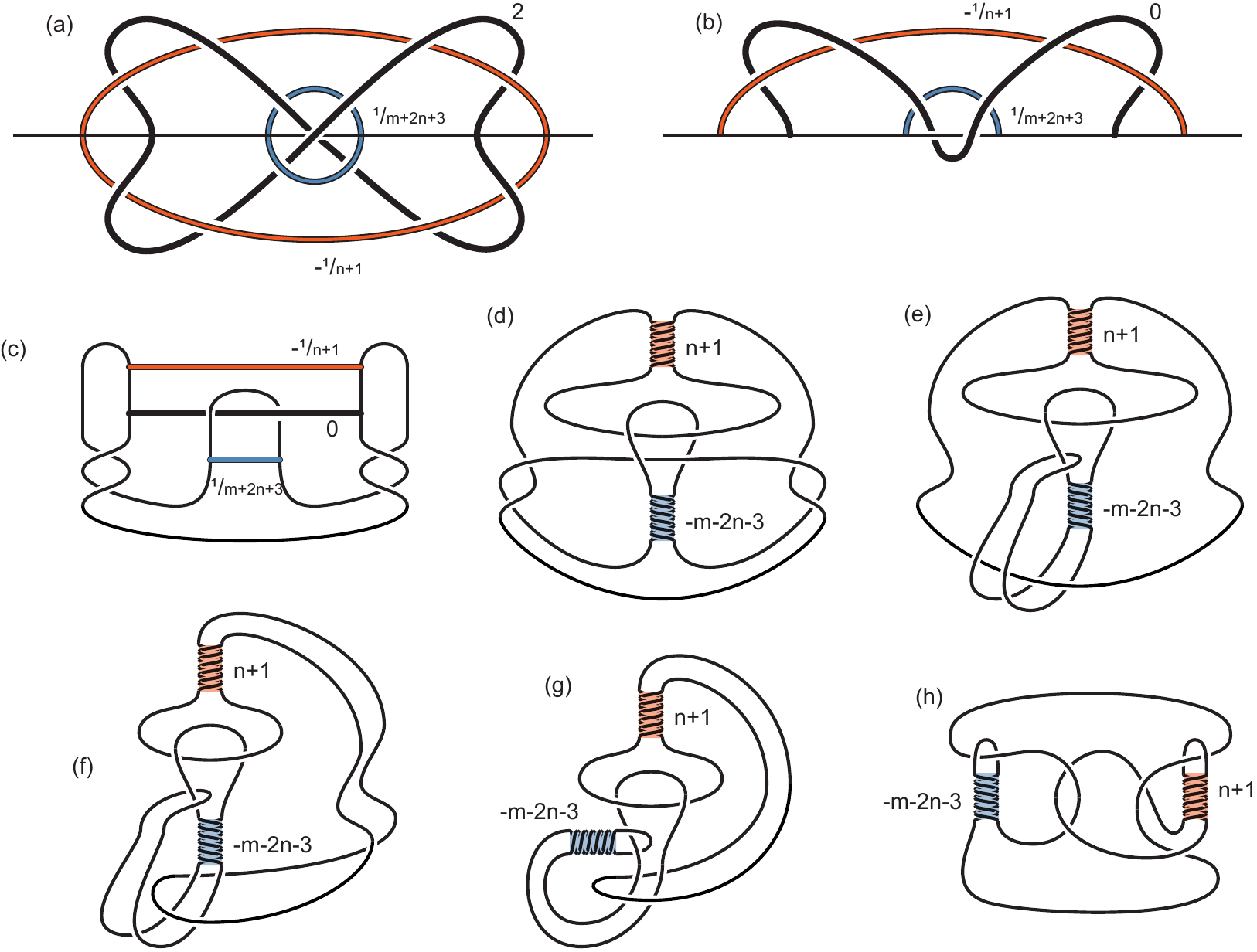}
    \caption{(a) The surgery description from Figure~\ref{fig:gen-surgdesc}(e) is strongly invertible. (b)\&(c) The quotient of the surgery description followed by some isotopy to straighten the arcs. (d) Rational tangle replacements along the arcs produce a link whose double branched cover is $-4m+8n+4$ surgery on $K_{n,m}$. (e)--(h) A sequence of isotopies show this link is the Montesinos link $M([0,-m-2n-3,-2],[0,-2], [0,1,-1,n+1,2]) = M(\frac{2}{2m+4n+5}, \frac{1}{2}, -\frac{2n+3}{4n+4})$.}
    \label{fig:gen-monty}
\end{figure}

\begin{figure}
    \centering
    \includegraphics[width=.7\textwidth]{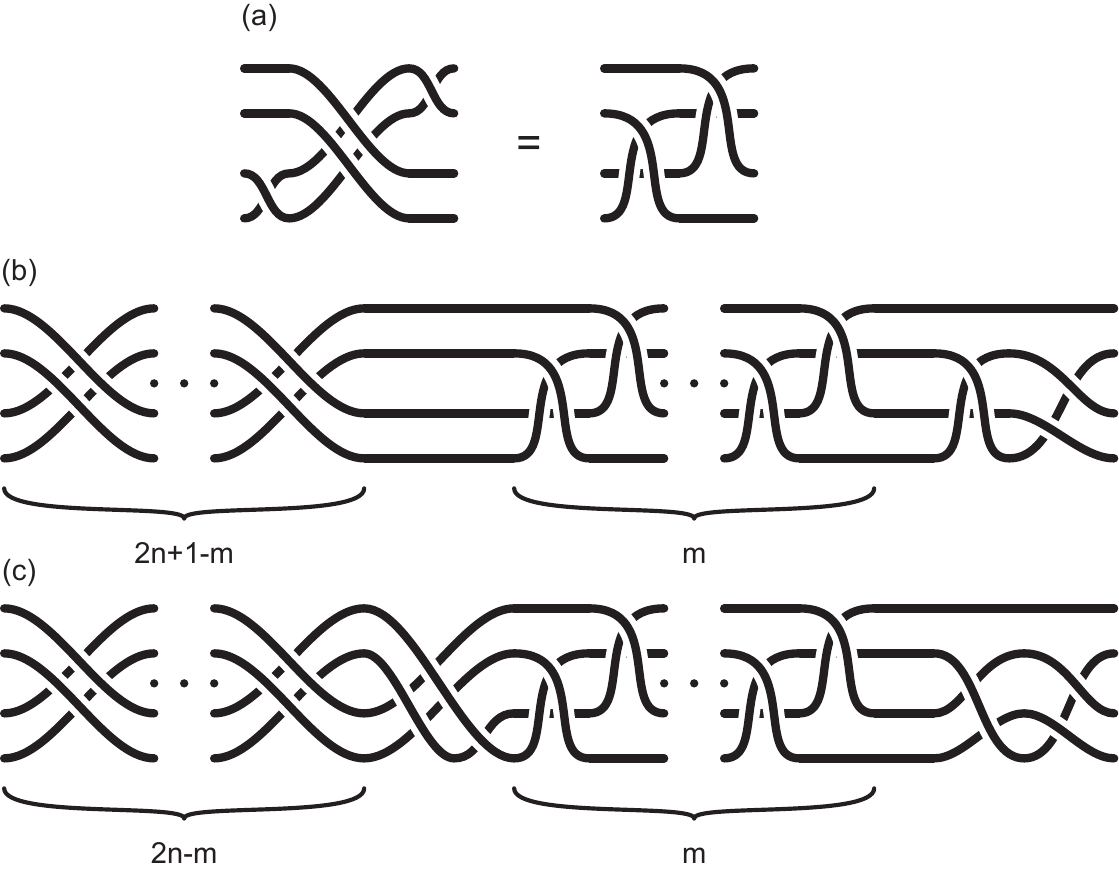}
    \caption{The proof of Proposition~\ref{prop:Knm}~(\ref{sqpgenusfibered})}
    \label{fig:SQPbraids}
\end{figure}

From our description of the family of knots $K_n$ in Figure~\ref{fig:surgdesc}, one finds a natural $2$--parameter family generalization.  While one may initially hope this family yields further examples of hyperbolic L-space knots that fail to be braid positive, we show this is not the case.

\begin{proposition}\label{prop:Knm}
Let $\beta_{n,m}$ be the braid indicated in Figure~\ref{fig:gen-surgdesc}(a), and let $K_{n,m} = \widehat{\beta}_{n,m}$ be its closure.
\begin{enumerate}
    \item\label{hyp} $K_{n,m}$ is a hyperbolic knot for all $(n,m) \in \Z^2$ except for the pairs 
    \[(n,m) \in \big\{(-1,k)\,\big\vert\, k\in\Z\big\}\cup\big\{(0,0),(0,-1),(0,-2),(-2,1),(-2,0),(-2,-1)\big\}.\] 
    For each of these pairs, $K_{n,m}$ is a torus knot.
    
    \item\label{sfssurg} $(8n+6-4m)$--surgery on $K_{n,m}$ gives the Seifert fibered space 
    \[M\left(-1; \frac12, \frac{2n+1}{4n+4}, \frac{2}{4n+5+2m}\right).\]

    \item\label{alex} The Alexander polynomial of $K_{n,m}$ is
    \[\Big(t^{m-1}\sum_{i=0}^{n} (t^{-4i-1} - t^{-4i})\Big) + 
    \Big( (-1)^m \sum_{j=-m}^{m} (-t)^j\Big) +
    \Big(t^{1-m}\sum_{k=0}^{n} (t^{4k+1}-t^{4k})\Big).\]

    \item\label{lsp}  Assume $n \geq 0$.  Then $K_{n,m}$ is an L-space knot if and only if $m\leq 0$.

    \item\label{posbraidgenus} If $n\geq 0$ and $m < 0$, then $\beta_{n,m}$ is a positive braid and  $K_{n,m}$ is a braid positive knot of genus $ |m|+4n+3$ 
    
    \item\label{sqpgenusfibered}  If  $2n+1\geq m \geq 0$, then $\beta_{n,m}$ is conjugate to a strongly quasipositive braid and $K_{n,m}$ is a strongly quasipositive knot of genus $4n-m+2$. 
    \begin{enumerate}
        \item If $2n\geq m\geq 0$ then $K_{n,m}$ is a fibered strongly quasipositive knot.    Moreover it is a Hopf plumbing basket. 
        
        \item If $2n+1= m > 0$ then $K_{n,m}$ is a non-fibered strongly quasipositive knot.
    \end{enumerate}

\end{enumerate}
\end{proposition}


\begin{proof}
(\ref{hyp}) Since the surgery description of $K_{n,m}$ given in Figure~\ref{fig:gen-surgdesc}(e) is on a hyperbolic link, using the $2\pi$ Theorem a couple of times yields a finite list of pairs $(n,m)$ for which $K_{n,m}$ might not be hyperbolic.  A further check in SnapPy confirms that all but 5 of them are hyperbolic.  These remaining 5 are readily confirmed to be torus knots. The computations are displayed at~\cite{BK}. 

\medskip

(\ref{sfssurg}) Figure~\ref{fig:gen-surgdesc} shows how to obtain a surgery description on a 3-component link for $-4m+8n+6$ surgery on $K_{n,m}$.   Figure~\ref{fig:gen-monty} uses the Montesinos trick to exhibit the result of this surgery description as the double branched cover of the Montesinos link $M([0,-2,-m-2n-3],[0,-2],[0,1,-1,n+1,2])$. 
This double branched cover is the Seifert fibered space $M(\frac12, -\frac{2n+3}{4n+4}, \frac{2}{4n+5+2m})$ which is equivalent to $M(-1; \frac12, \frac{2n+1}{4n+4}, \frac{2}{4n+5+2m})$.

\medskip

(\ref{posbraidgenus})  When $n\geq0$ and $m<0$, the braid $\beta_{n,m}$ as described in Figure~\ref{fig:gen-surgdesc}(a) is expressly a positive braid.  One counts that it is a braid of index four and $4(2n+1)+(1-2m)+4$ crossings.  Hence $\chi(K_{n,m}) = -(2|m|+8n+5)$ and $g(K_{n,m}) = |m|+4n+3$.

\medskip

(\ref{sqpgenusfibered})   
When $0\leq m \leq 2n+1$, through braid isotopy and braid conjugacy, we may isotop in pairs $2m$ of the $2m+1$ negative crossings over to $m$ of the $2n+1$ copies of the ``2-cabled'' positive crossing that appear in $\beta_{n,m}$ so that they appear as in the left hand side of Figure~\ref{fig:SQPbraids}(a).  Hence by a further braid isotopy as indicated by Figure~\ref{fig:SQPbraids}, each of these $2m$ negative crossings contributes to a SQP band.  The final negative crossing also contributes to a SQP band towards the end of the braid, ultimately giving us the strongly quasipositive braid shown in Figure~\ref{fig:SQPbraids}(b) to which $\beta_{n,m}$ is conjugate.
One counts that the braid index is $4$ and there are $2m+1$ SQP bands and $4(2n+1-m)+2$ regular crossings.  Hence $\chi(K_{n,m}) = -(8n-2m+3)$  and $g(K_{n,m}) = 4n-m+2$.



Furthermore, when $0\leq m \leq 2n$ so that $2n-m \geq 0$, we may instead perform braid isotopy and conjugation to arrive at the strongly quasipositive braid shown in Figure~\ref{fig:SQPbraids}(c).  This braid however contains the `dual Garside element' $\delta = \sigma_3 \sigma_2 \sigma_1$. Hence, as Banfield points out~\cite{banfield2016strongly}, the closure of such a SQP braid is fibered and a Hopf basket.  

When $m=2n+1$, we have that the braid $\beta_{n,2n+1}$ is conjugate to a SQP braid but its closure $K_{n,2n+1}$ might not be fibered.  Indeed, we find that the Alexander polynomial of $K_{n,2n+1}$ is not monic, so the closure is not fibered.
Explicitly, from our computations of $\Delta_{K_{n,m}}$ for item (\ref{alex}) below we have
\begin{align*}
    \Delta_{K_{n,2n+1}}(t) &= \frac{t-1}{(t^4-1)(t^2-1)} \cdot t(t^2-1)(2-t+t^2+t^{4n+3}-t^{4n+4}+2t^{4n+5})\\
    &=t\frac{(2t^{4n+6}-2t^2)-(3t^{4n+5}-3t)+ (2t^{4n+4}-2)-(t^{4n+3}-t^3) }{t^4-1}\\
    &\doteq  \frac{(2-3t+2t^2)(t^{4(n+1)}-1)-t^3(t^{4n}-1)}{t^4-1} \\
    &=\frac{(2-3t+2t^2)(t^{4(n+1)}-t^{4n}+t^{4n}-1)-t^3(t^{4n}-1)}{t^4-1}\\ 
    &=(2-3t+2t^2)t^{4n}+(2-3t+2t^2-t^3)\frac{t^{4n}-1}{t^4-1}
\end{align*}
which has leading coefficient $2$.

\medskip

(\ref{alex})  View the surgery description for $K_{n,m}$ as the link $L=K \cup c \cup c'$ where we do $\frac{-1}{n+1}$ surgery on $c$ and $\frac{1}{m+2n+3}$ surgery on $c'$.
Observe $c \cup c'$ is the trivial $2$--component link and we may orient the link so that $\lk(K,c)=4$ and $\lk(K,c')=2$.

Let $E$ be the exterior of $L=K \cup c \cup c'$.  Then $H_1(E) = \langle [\mu_K], [\mu_c], [\mu_{c'}]\rangle \cong \Z^3$ where $\mu_K$, $\mu_c$, and $\mu_{c'}$ are oriented meridians of $K$, $c$, and $c'$.  Let $\lambda_K$, $\lambda_c$, and $\lambda_{c'}$ be their preferred longitudes.  Observe that $[\lambda_c] = 4[\mu_K]$ and $[\lambda_{c'}] = 2[\mu_K]$ in $H_1(E)$.

Now consider the family of links $L_{n,m} = K_{n,m} \cup c_n \cup c'_m$ with exterior $E_{n,m}$ obtained from $K$ and the core curves of $\frac{-1}{n+1}$ surgery on $c$ and $\frac{1}{m+2n+3}$ surgery on $c'$.  Thus $E_{n,m} \cong E$ where 
\[ \mu_{K_{n,m}} = \mu_K, \quad  \mu_{c_n} = -\mu_c + (n+1)\lambda_c, \quad \mbox{ and } \quad \mu_{c'_m}= \mu_{c'} +(m+2n+3)\lambda_{c'}.\]
Now letting
\begin{align}
x&=[\mu_K] & y&=[\mu_c] & z&=[\mu_{c'}]\\
x_{n,m}&=[\mu_{K_{n,m}}] & y_n&=[\mu_{c_n}] & z_m&=[\mu_{c'_m}]
\end{align}
in the group rings $\Z[H_1(E)]$ and $\Z[H_1(E_{n,m})]$
we have 
\[ {x_{n,m}} = x, \quad  y_n = y^{-1} x^{4(n+1)}, \quad \mbox{ and } \quad z_m= zx^{2(m+2n+3)}\]
and hence
\[x=x_{n,m}, \quad y = y_n^{-1}x_{n,m}^{4(n+1)}, \quad \mbox{ and } \quad z = z_m x_{n,m}^{-2(m+2n+3)}.\]
Therefore
\begin{equation}\label{eqn:alexshift}
    \Delta_{L_{n,m}}(x_{n,m}, y_n, z_m)= \Delta_L(x_{n,m}, y_n^{-1}x_{n,m}^{4(n+1)}, z_m x_{n,m}^{-2(m+2n+3)}).
\end{equation} 
Using the Torres Formulae \cite{torres}, one obtains that
\begin{align}\label{eqn:torres}
    \Delta_{K_{n,m}}(x_{n,m}) &= \frac{x_{n,m}-1}{x_{n,m}^4-1} \cdot \Delta_{K_{n,m} \cup c_n}(x_{n,m}, 1) \\
    &=\frac{x_{n,m}-1}{(x_{n,m}^4-1)(x_{n,m}^2-1)} \cdot \Delta_{K_{n,m} \cup c_n \cup c'_m}(x_{n,m}, 1,1). \nonumber
\end{align}
Hence, using Equations (\ref{eqn:alexshift}) and (\ref{eqn:torres}) where we set $x_{n,m} = t$, $y_{n} = 1$ and $z_m = 1$, we obtain
\[\Delta_{K_{n,m}}(t) = \frac{t-1}{(t^4-1)(t^2-1)} \cdot \Delta_{L}(t, t^{4(n+1)}, t^{-2(m+2n+3)}).\]
We calculate that  
\begin{equation*}
    \Delta_L(x,y,z)=(x^2-1)(x^3 y^2 z+x^2 y^3 z-x^2 y^2 z+x^2 y+x y^2 z-x y+x+y).
\end{equation*}
Then 
\begin{align*}
 \Delta_{K_{n,m}}(t) &=  \frac{t-1}{(t^4-1)(t^2-1)} \cdot   \Delta_L(t, t^{4(n+1)},t^{-2(m+2n+3)} ) \\
    &= t^{4n+3-m} \frac{(t-1)(t^{m-4n-2}+t^{-m}-t^{1-m}+t^{2-m}+t^{m+1}-t^{m+2}+t^{m+3}+t^{-m+4n+5})}{(t^4-1)}\\       
    &\doteq \frac{(t-1)\Big((t^{m-4n-2}-t^{m+2})+(t^{-m}+t^{2-m}+t^{m+1}+t^{m+3})+(t^{4n+5-m}-t^{1-m})\Big)}{(t^4-1)}\\   
    &= \frac{t^{m+2}(t-1)(t^{-4n-4}-1)}{t^4-1}+
    \frac{(t-1)(t^{-m}+t^{2-m}+t^{m+1}+t^{m+3})}{t^4-1}+
    \frac{t^{1-m}(t-1)(t^{4n+4}-1)}{t^4-1}\\
    &= \Big(t^{m-1}\sum_{i=0}^{n} t^{-4i} (t^{-1}-1)\Big) + 
    \Big( \frac{t^{m+1}-t^m+t^{-m}-t^{-m-1}}{t-t^{-1}}\Big) +
    \Big(t^{1-m}\sum_{j=0}^{n} t^{4j} (t-1)\Big)\\
    &= \Big(t^{m-1}\sum_{i=0}^{n} (t^{-4i-1} - t^{-4i})\Big) + 
    \Big( (-1)^m \sum_{j=-m}^{m} (-t)^j\Big) +
    \Big(t^{1-m}\sum_{k=0}^{n} (t^{4k+1}-t^{4k})\Big)\\
\end{align*}
where the $\doteq$ indicates that we have divided out the unit $t^{4n+3-m}$.  

\medskip

(\ref{lsp})
Using our Alexander polynomial calculations provide obstructions to the knots $K_{n,m}$ for $n>0$ being L-space knots when $m>0$.
As an example, taking $n>0$ and $m=1$ gives
\begin{align*}
 \Delta_{K_{n,1}}(t) &=  \frac{t-1}{(t^4-1)(t^2-1)} \cdot   \Delta_L(t, t^{4(n+1)},t^{-2(2n+4)} ) \\
    & \doteq \Big(\sum_{i=0}^{n} (t^{4i-1}-t^{4i})\Big) + \Big(t^{-1}-1+t\Big) + \Big(\sum_{k=0}^{n}  (t^{4k+1}-t^{4k})\Big).\\
\end{align*}
One may observe that the constant coefficient is $-3$.   Hence the knots $K_{n,1}$ cannot be L-space knots. Indeed, one may further observe that, when $n>0$ and $m>0$, the central terms will have overlap with the end terms to give coefficients $\pm2$ or $\pm3$ for terms with degree of small magnitude.  Thus none of the knots $K_{n,m}$ with $n>0$ and $m>0$ are L-space knots.

In the other direction, where $n>0$ and $m \leq 0$, we may observe via~\cite{LS,LM} as in Lemma~\ref{lem:Lspaceknots} that the Seifert fibered space $M$ resulting from $(8n+6-4m)$--surgery on $K_{n,m}$ is an L-space.  
For that we need to distinguish several cases. We continue with the notation of Lisca-Stipsicz~\cite{LS} as in Lemma~\ref{lem:Lspaceknots}.

Since $n>0$, we have 
\begin{equation*}
    1>\frac 12>\frac{2n+1}{4n+4}>0.
\end{equation*}
So we must reckon with the coefficient 
\begin{equation*}
    \frac{2}{2m+4n+5}=\frac{2}{2(2n+m+1)+3}.
\end{equation*}
If $2n+m+1\geq1$ we see that 
\begin{equation*}
    1>\frac 12>\frac{2n+1}{4n+4}>\frac{2}{2m+4n+5}>0.
\end{equation*}
If we now assume that there exist coprime integers $a, b$ such that 
\begin{equation*}
    \frac12 b<a<\frac{2n+3}{4n+4}b \,\,\text{ and }\,\, \frac{2}{4n+2m+5}b<1
\end{equation*}
we conclude from the first inequality that $0<2a-b<\frac{b}{2n+2}$ and the second inequality implies that $b\leq 2n+2+m\leq2n+2$. Putting both together yields the contradiction
\begin{equation*}
    0<2a-b<\frac{b}{2n+2}\leq1.
\end{equation*}
Thus $M$ carries no positive transverse contact structure and is therefore an L-space.

If $2n+m+1=0$ we get the Seifert fibered space $M(-1;\frac23,\frac12,\frac{2n+1}{4n+4})$. We assume that there exist coprime integers $a,b$ such that $\frac23 b <a<\frac12 b$ and $\frac{2n+1}{4n+4}b<1$ from which we conclude $4b<6a<3b$ and $b<2+\frac{2}{2n+1}\leq4$, which is a contradiction. Therefore $M$ dos not carry a positive transverse contact structure and is thus an L-space.

If $2n+m+1=-1$, we get the Seifert fibered space $M(-1;\frac12,\frac{2n+1}{4n+4},2)=M(1;\frac12,\frac{2n+1}{4n+4})$ which is a lens space and hence an L-space.

If $2n+m+1=-2$ we consider the Seifert fibered space $M(-1;\frac12,\frac{2n+1}{4n+4},-2)=M(-3;\frac12,\frac{2n+1}{4n+4})$ which is a lens space and hence an L-space.


If $2n+m+1\leq -3$, we see that 
\begin{equation*}
    \frac{2}{2m+4n+5}=\frac{2}{2(2n+m+1)+3}\in[-1,0]
\end{equation*}
and thus the correctly normalized Seifert fibered space is $M(-2;\frac12,\frac{2n+1}{4n+4},\frac{4n+2m+7}{4n+2m+5})$ which admits a positive contact structure. Next, we consider 
$$-M=M\left(2;-\frac12,-\frac{2n+1}{4n+4},-\frac{4n+2m+7}{4n+2m+5}\right)=M\left(-1;\frac12,\frac{2n+3}{4n+4},-\frac{2}{4n+2m+5}\right).$$ If now $2n+m+1=-3$, then the correct ordering of the Seifert invariants is $M(-1;\frac23,\frac{2n+3}{4n+4},\frac12)$. We readily see that there exist no coprime integers $a,b$ such that $\frac23 b<a<\frac{2n+1}{4n+4} b$ and $\frac12 b<1$. Thus $M$ carries no positive transverse contact structure and is therefore an L-space. If $2n+m+1\leq-4$ the Seifert invariants are ordered as $M(-1;\frac{2n+3}{4n+4},\frac 12,-\frac{2}{4n+2m+5})$. We assume that there exist coprime integers $a,b$ such that 
\begin{equation*}
    \frac{2n+3}{4n+4} b<a<\frac12 b\,\,\text{ and }\,\, -\frac{2}{4n+2m+5} b<1.
\end{equation*}
But putting them together yields the contradiction
\begin{equation*}
0<a-\frac{2n+3}{4n+4} b<-\frac{1}{4n+4} b<0.
\end{equation*}
Thus $M$ does not admit a positive transverse contact structure and is therefore an L-space.
\end{proof}

\begin{rem}
In the cases of the above proof when $2n+m+1 =-1$ or $-2$, the knots $K_{n,m}$ have lens space surgeries. These knots  can be seen to be Berge knots as follows. With $-m-2n-3 = 1$ or $0$, Figure~\ref{fig:gen-monty}(d) can be seen to divide along a horizontal line into two rational tangles.  A vertical arc in the middle would be the arc dual to the rational tangle replacement on the $0$--framed arc from Figure~\ref{fig:gen-monty}(c).  In the double branched cover, this vertical arc will lift to a knot in the lens space with an $S^3$ surgery. Furthermore, one may observe that this arc lifts to a $(1,1)$--knot in the lens space.  Hence the knot $K_{n,m}$ must be a Berge knot \cite{berge}.
\end{rem}

\section{Curiosities about the Alexander polynomial of $o9\_30634$ and its generalizations}\label{sec:curiosities}

Like the failure of braid positivity, the hyperbolic L-space knot $o9\_30634$ exhibits two more curious properties that had previously only been observed for L-space knots among iterated cables of torus knots. The first, regarding formal semigroups, Teragaito communicated to us near the completion of the initial preprint.  The second, regarding the roots of its Alexander polynomial, came after that. Both actually generalize to the infinite family $\{K_n\}_{n \geq 1}$ as well.

\subsection{An infinite family of hyperbolic L-space knots whose formal semigroups are semigroups}

Teragaito informed us about the work of Wang \cite{Wang18} on formal semigroups of L-space knots and that there are only two L-space knots in the SnapPy census whose formal semigroups were actual semigroups.  He had also observed that one of these knots appeared to fail to be braid positive.  It turns out that this is the knot $o9\_30634$ which we had confirmed to not be braid positive.  Upon seeing an early draft of this article, Teragaito further showed that all of our hyperbolic L-space knots $K_n$ have formal semigroups that are semigroups.  Below we overview the formal semigroup and then record Teragaito's results in Theorem~\ref{thm:formal_semigroups}.

An algebraic link is defined to be the link of an isolated singularity of a complex curve in $\C^2$. Algebraic knots are known to be iterated cables of torus knots~\cite{EisenbudNeumann} and they are all L-space knots, see \cite{heddencabling2}. Moreover, one can assign to any algebraic knot $K$ an additive semigroup $S_K< \N_0$ which determines the Heegaard Floer chain complex and is computable from the Alexander polynomial of $K$, see~\cite{BodnarCeloriaGolla} and references in there. 

In~\cite{Wang18} Wang has generalized this definition, but now $S_K$ is not necessarily a semigroup anymore. Let $K$ be an L-space knot with (symmetrized) Alexander polynomial $\Delta_K$ then the \textit{formal semigroup} $S_K\subset\N_0$ is defined by
\begin{equation*}
    \frac{t^{g(K)}\Delta_K(t)}{1-t}=\sum_{s\in S_K}t^s
\end{equation*}
where $g(K)$ denotes the genus of $K$. (Note that $t^{g(K)}\Delta_K(t)$ is now an ordinary polynomial of degree $2g(K)$.) The set $S_K$ still determines the Heegaard Floer chain complex of $K$ but is not necessarily a semigroup anymore. This is used by Wang to construct an infinite family of L-space knots which are iterated cables of torus knots but not algebraic~\cite{Wang18}. On the other hand, it remained open if there exist an L-space knot which is not an iterated cable of torus knots but whose formal semigroup is a semigroup~\cite[Question~2.8]{Wang18}.

\begin{theorem}[Teragaito \cite{teragaito}]
\label{thm:formal_semigroups}
There exist an infinite family of hyperbolic L-space knots whose formal semigroups are semigroups. More concretely:
\begin{enumerate}
    \item\label{census} $o9\_30634$ and $t09847$ are hyperbolic L-space knots whose formal semigroups are semigroups. The formal semigroup of every other L-space knot in the SnapPy census is not a semigroup. 

\item\label{infinite} The formal semigroups $S_{K_n}$ of the infinite family of hyperbolic L-space knots $\{K_n\}$ from Section~\ref{sec:potential_counterexamples} are all semigroups.
\end{enumerate}
Consequently, the knots $\{K_n\}$ provide an infinite family of knots giving a negative answer to \cite[Question 2.8]{Wang18}.
\end{theorem}

\begin{proof}
(\ref{census}) The formal semigroup $S_K$ of an L-space knot is computable from the Alexander polynomial of $K$, in particular we see that $S_K$ always contains all natural numbers larger than $g(K)$ and the finitely many other elements of $S_K$ can be read-off from the Alexander polynomial. In~\cite{BK} we present code that computes the formal semigroups of all SnapPy census L-space knots and determines that $o9\_30634$ and $t09847$ are the only ones whose formal semigroups are semigroups.

(\ref{infinite}) In Proposition~\ref{prop:Knm}(\ref{alex}) we have computed the Alexander polynomials of $K_n$ from which we read-off the formal semigroup to be
\begin{align*}
    S_{K_n}=&\{0,4,8,\ldots,4n\}\cup\{4n+2\}\cup\{4n+4,4n+5,4n+6,4n+8,4n+9,4n+10,\\
    &4n+12,4n+13,4n+14,\ldots,8n,8n+1,8n+2,8n+4\}\cup \N_{>8n+4},
\end{align*}
which is a semigroup for any $n$.
\end{proof}

\begin{rem}[Teragaito \cite{teragaito}]
A braid word of $t09847$ is given by 
\begin{equation*}
    [(2,1,3,2)^{3}, 1,2,1,1,2],
\end{equation*}
which is very close to our braid word for $o9\_30634$. One can similarly show that $t09847$ fits into an infinite family of hyperbolic L-space knots with braid words
\begin{equation*}
    [(2,1,3,2)^{2n+1}, 1,2,1,1,2],
\end{equation*}
whose formal semigroups are semigroups.
\end{rem}

\begin{rem}\label{rem:semigroup}
The semigroups from Theorem~\ref{thm:formal_semigroups} and the preceeding remark all have rank $3$, i.e. the minimal number of a generating set is $3$. On the other hand, Teragaito constructs in~\cite{teragaitob} an infinite family of hyperbolic L-space knots whose formal semigroups are semigroups of rank $5$.
\end{rem}

\subsection{Two infinite families of hyperbolic L-space knots whose Alexander polynomial roots are all roots of unity}

The Alexander polynomial of $o9\_30634 = K_1 =K_{1,0}$ can be written as 
\[\Delta_{o9\_30634}(t) \doteq \frac{\left(t^{6}+1\right) \left(t^{9}+1\right)}{(t+1) \left(t^2+1\right)}.\]
From this one may observe that all of its roots are roots of unity.  Since $o9\_30634$ is a hyperbolic L-space knot, it provides a counterexample to \cite[Conjecture 1.3]{LiNi} (see also the discussion surrounding its reference as \cite[Conjecture 6.10]{HWbotanygeography}).
Indeed, we have infinite families of hyperbolic L-space knots that are counterexamples to this conjecture.

\begin{theorem}\label{thm:rootsofunity}
The two infinite families of hyperbolic L-space knots $\{K_n\}_{n\geq1}$ and $\{K_{n,-1}\}_{n \geq 1}$ consist of knots whose Alexander polynomials have all of their roots on the unit circle.  
\end{theorem}

\begin{proof}
Proposition~\ref{prop:Knm}(\ref{hyp}) and (\ref{lsp}) show that the knots of $\{K_n\}_{n\geq1}$ and $\{K_{n,-1}\}_{n \geq 1}$ are hyperbolic L-space knots. 
Proposition~\ref{prop:Knm}(\ref{alex}) gives a general formula for $\Delta_{K_{n,m}}(t)$.  In the course of that proof, we obtained the first equality below.  Dividing out the unit $t$ and rearranging gives the second.
\begin{align*}
 \Delta_{K_{n,m}}(t) 
    &= t^{4n+3-m} \frac{(t-1)(t^{m-4n-2}+t^{-m}-t^{1-m}+t^{2-m}+t^{m+1}-t^{m+2}+t^{m+3}+t^{-m+4n+5})}{(t^4-1)}\\
    & \doteq \frac{\left(t^{8 n+7}+t^{4n+4}-t^{4n+3}+t^{4n+2}\right) t^{-2 m}+\left(t^{4 n+3}-t^{4 n+4}+t^{4 n+5}+1\right)}{(t+1) \left(t^2+1\right)}\\
        \end{align*}
Setting $m=0$ yields
\begin{align*}
 \Delta_{K_{n,0}}(t) 
    & \doteq \frac{\left(t^{8 n+7}+t^{4n+4}-t^{4n+3}+t^{4n+2}\right)+\left(t^{4 n+3}-t^{4 n+4}+t^{4 n+5}+1\right)}{(t+1) \left(t^2+1\right)} \\
    &= \frac{t^{8 n+7}+t^{4 n+5}+t^{4n+2} +1}{(t+1) \left(t^2+1\right)}
    = \frac{\left(t^{4 n+5}+1\right)\left(t^{4n+2} +1\right)}{(t+1) \left(t^2+1\right)}
\end{align*}
while setting $m=-1$ yields
\begin{align*}
\Delta_{K_{n,-1}}(t) 
    &\doteq \frac{\left(t^{8 n+9}+t^{4n+6}-t^{4n+5}+t^{4n+4}\right)+\left(t^{4 n+3}-t^{4 n+4}+t^{4 n+5}+1\right)}{(t+1) \left(t^2+1\right)} \\
    &=\frac{t^{8 n+9}+t^{4n+6}+t^{4 n+3}+1}{(t+1) \left(t^2+1\right)} 
    = \frac{\left(t^{4n+6}+1\right)\left(t^{4 n+3}+1\right)}{(t+1) \left(t^2+1\right)}.
\end{align*}
From these presentations of their Alexander polynomials, one sees that all of their roots are roots of unity.
\end{proof}

\begin{rem} \phantom{x}
\begin{enumerate}
    \item While we do not yet know if any of the knots in $\{K_n\}_{n\geq1}$ are braid positive, all of the knots  $\{K_{n,-1}\}_{n\geq1}$ are braid positive by Proposition~\ref{prop:Knm}(\ref{posbraidgenus}).
    \item As one may check, the hyperbolic L-space knots $\{K_{n,-2}\}_{n\geq1}$ have Alexander polynomials with roots that are not roots of unity.
\end{enumerate}
\end{rem}

\begin{rem}
In light of Theorem~\ref{thm:rootsofunity} and \cite[Corollary 1.2]{BBG}, one may hope that that at least one the hyperbolic L-space knots among $\{K_n\}_{n\geq1}$ and $\{K_{n,-1}\}_{n \geq 1}$ has a double branched cover that is an L-space.  This would answer a question of Moore in the negative; see \cite[Question 1.3]{BBG}.  However, as one may check, these knots are not {\em definite}.  Indeed, $|\sigma(K_n)| = g(K_n)+2 < 2g(K_n)$ while $|\sigma(K_{n,-1})| = g(K_{n,-1})+3 <2g(K_{n,-1})$.
\end{rem} 

\subsection*{Acknowledgements}
We thank Tetsuya Ito for comments and discussions about braid positivity.  We thank Masakazu Teragaito for sharing his investigations on the formal semigroup and allowing their inclusion here. We thank Neil Hoffman for his help with handling orientation issues in SnapPy. We thank Steven Sivek for pointing out the potential for examples of hyperbolic L-space knots whose double branched covers are L-spaces.

KLB was partially supported by a grant from the Simons Foundation (grant \#523883  to Kenneth L.\ Baker).

MK thanks ICERM (the Institute for Computational and Experimental Research in Mathematics in Providence, RI)
for the productive environment during the semester program on Braids (Feb 1 - May 6, 2022) where part of this work was carried out. He is also thankful to the first author and
the University of Miami for their hospitality during his 2022 visit.

\let\MRhref\undefined
\bibliographystyle{hamsalpha}
\bibliography{NonBraidPositive.bib}

\clearpage
\section{Appendix}\label{appendix}
In this appendix we display braid words for all SnapPy census  knots. Except for $o9\_30634$, the braid words of the L-space knots are positive or negative (if the chirality in the SnapPy census yields a negative L-space knot).
	\begin{adjustwidth}{}{}
{\tiny
\begin{tabularx}{\linewidth}{@{}p{1cm}Xc>{\centering\arraybackslash}p{1.6cm}c@{}}

        \toprule
        \textbf{L-space knot} & \textbf{Braid word} \\
      \midrule
$m016$ & [-1, -1, -1, -2, -1, -2, -2, -2, -2, -1, -2, -2] \\
 $m071$ & [1, 2, 2, 3, 2, 4, 3, 3, 2, 1, 3, 4, 3, 2, 1, 3, 2, 3, 4, 3, 3, 2, 1, 2, 3, 4] \\
 $m082$ & [1, 2, 2, 1, 2, 2, 2, 1, 1, 2, 1, 1, 2, 1, 1, 1, 1, 1] \\
 $m103$ & [-1, -2, -3, -2, -1, -2, -3, -3, -4, -3, -2, -2, -2, -2, -3, -2, -2, -3, -2, -1, -2, -3, -3, -2, -4, -3, -2, -2, -3, -2, -4, -1, -2, -1] \\
 $m118$ & [1, 2, 2, 1, 3, 3, 3, 2, 1, 3, 3, 2, 1, 3, 2, 3, 2, 2, 1, 3, 3, 2, 3] \\
 $m144$ & [-1, -2, -1, -2, -1, -2, -1, -2, -1, -1, -2, -2, -1, -2, -2, -1, -1, -1, -2, -1, -2, -2, -2, -1] \\
 $m194$ & [-1, -2, -2, -2, -3, -2, -4, -2, -3, -2, -3, -2, -1, -2, -3, -2, -3, -4, -3, -2, -1, -3, -2, -4, -1, -2, -2, -3, -2, -3] \\
\bottomrule

    \end{tabularx}
       	}
       	\end{adjustwidth}

\begin{adjustwidth}{}{}
{\tiny
    \begin{tabularx}{\linewidth}{@{}p{1cm}Xc>{\centering\arraybackslash}p{1.6cm}c@{}}

        \toprule
        \textbf{non L-space knot} & \textbf{Braid word} \\
      \midrule
$m004$ &  [1, -2, 1, -2] \\
 $m015$ &  [-1, 2, 3, 2, 1, 2, 2, -3, 2] \\
 $m032$ &  [-1, -2, 3, -2, -4, 1, 3, 2, 3, 4, 3, -2] \\
 $m053$ &  [-1, 2, 3, 4, -5, 4, 3, -2, 1, 3, -4, 3, 5, 3, -2, 3, -4, 3, 2] \\
 $m074$ & 
  [-1,
   -2,
   3,
   -4,
   3,
   5,
   2,
   -6,
   1,
   5,
   -4,
   3,
   -4,
   -5,
   4,
   6,
   3,
   -2,
   3,
   -4,
   -5,
   4,
   3,
   2] \\
 $m094$ & 
  [-1,
   2,
   -3,
   4,
   -5,
   -6,
   5,
   -7,
   -6,
   5,
   4,
   3,
   -2,
   1,
   -3,
   4,
   -5,
   4,
   6,
   3,
   7,
   -2,
   -3,
   4,
   4,
   -5,
   4,
   6,
   3,
   2,
   4,
   5,
   4] \\
 $m199$ &  [1, -2, 3, -2, -3, 2, 2, -1, 2, -3, 2] \\
\bottomrule

    \end{tabularx}
       	}
       	\end{adjustwidth}

\end{document}